\documentclass[10pt]{amsart}
\usepackage{graphicx}
\usepackage{latexsym}
\usepackage{fancyhdr}
\usepackage{amsmath, amssymb}
 \usepackage[utf8]{inputenc}
\usepackage[all]{xy}
\usepackage{pdflscape}
\usepackage{longtable}
\usepackage{rotating}
\usepackage{hyperref}
\usepackage{subfigure}
\usepackage{tensor}

\usepackage{enumerate}

\usepackage{color}

\theoremstyle{plain}
\newtheorem{thm}{Theorem}

\newtheorem{cor}[thm]{Corollary}

\newtheorem{lem}[thm]{Lemma} 
\newtheorem{prop}[thm]{Proposition}

\theoremstyle{definition}

\newtheorem{rem}[thm]{Remark}
\numberwithin{equation}{section}

\newcommand{\lgw}{\longrightarrow}
\newcommand{\lgm}{\longmapsto}
\newcommand{\si}{\sigma}

\newcommand{\ovl}{\overline}

\renewcommand{\deg}{\text{deg}\,}
\newcommand{\ord}{\text{ord}}

\newcommand{\uu}{\mathbf u}
\newcommand{\vv}{\mathbf v}

\newcommand{\wdt}{\widetilde}

\newcommand{\D}{\Delta}

\newcommand{\la}{\lambda}

\newcommand{\R}{\mathbb{R}}
\newcommand{\K}{\mathbb{K}}

\newcommand{\N}{\mathbb{N}}

\renewcommand{\L}{\mathbb{L}}
\newcommand{\C}{\mathbb{C}}

\newcommand{\Q}{\mathbb{Q}}
\renewcommand{\t}{\tau}
\renewcommand{\tt}{\mathbf{t}}

\renewcommand{\lg}{\langle}
\newcommand{\rg}{\rangle}
\newcommand{\q}{\mathbf{q}}
\renewcommand{\a}{\alpha}

\renewcommand{\b}{\beta}
\newcommand{\g}{\gamma}

\renewcommand{\phi}{\varphi}

\let\mathscr\mathcal

\begin{document}
\title[Topological local algebraicity]{Local topological  algebraicity with algebraic coefficients of analytic sets or functions}
\author{Guillaume Rond}
\email{guillaume.rond@univ-amu.fr}
\address{Aix Marseille Univ, CNRS, Centrale Marseille, I2M, Marseille, France}

\begin{abstract} We prove that any complex or real analytic set or function germ is topologically equivalent to  a germ defined by polynomial equations whose coefficients are algebraic numbers.
\end{abstract}


\maketitle

The problem of the algebraicity of analytic sets or mappings is an old subject of study. It is known that the germ of a coherent analytic set with an isolated singularity is analytically equivalent to the germ of an algebraic set \cite{Ku, To1}. But in the general case the germ of an analytic set is not even locally diffeomorphic to the germ of an algebraic set \cite{Wh}. On the other hand, considering a weaker equivalence relation,  T. Mostowski proved that the germ of an analytic set is always homeomorphic to the germ of an algebraic set \cite{Mo} and this has been generalized to analytic function germs \cite{BPR}. \\
For practical, effective and sometimes even theoretical purposes (for instance see \cite{BW}) it is often not  possible to handle coefficients that are  transcendental numbers and so it is important to work with polynomial equations whose coefficients are  rational or  algebraic numbers. But it is well known that a small perturbation of the coefficients of polynomial equations defining an algebraic set germ or an algebraic function germ can drastically change the topology of the  germ.  \\
\\
The goal of this paper is to extend the results of \cite{BPR}  by proving that any complex or real analytic set or function germ is homeomorphic to an algebraic  germ defined over the algebraic numbers. Our main result is the following one:

\begin{thm} \label{thm2}  
Let  $\K = \R$ or $\C$.  
 Let $(V,0) \subset (\K^n,0)$ be an analytic set germ and $g: (V,0)\to (\K, 0)$ be an analytic function germ.  \\
 Then there is a homeomorphism 
 $$h : (\K^n,0) \to (\K^n,0)$$ such that 
 \begin{enumerate}
 \item[i)] $(h(V),0)$ is the germ of an algebraic subset of $\K^n$ defined over $\ovl\Q\cap\K$,
 \item[ii)]  $g\circ h^{-1}$ 
 is the germ of a polynomial function defined over $\ovl\Q\cap\K$.  \\
 \end{enumerate}
\end{thm}

Moreover when we consider the particular case where there is no function germ $g$ but only the set germ $(V,0)$ we can be more precise about the nature of the homeomorphism:

\begin{thm} \label{thm1}  
Let  $\K = \R$ or $\C$.  
Let  $(V,0) \subset (\K^n,0)$ be an analytic set germ.  
 Then there is a homeomorphism $h: (\K^n,0) \to (\K^n,0)$ such that 
 \begin{enumerate}
 \item[i)] $h(V)$ 
 is the germ of an algebraic subset of $\K^n$ defined over $\ovl \Q\cap\K$, 
  \item[ii)] $(V,0)$ is Whitney equisingular with $(h(V),0)$,

\item[iii)] $h$ is subanalytic and arc-analytic.\\
 \end{enumerate}
\end{thm}
 
The proof of our main result is based on the approach introduced in \cite{Mo} and extended in \cite{BPR}. For instance the idea to prove Theorem \ref{thm1} in the case where $(V,0)$ is an hypersurface germ is to use a version of the Nested Artin-P\l oski-Popescu Approximation Theorem, that we prove in this paper (see Theorem \ref{nest_ploski}),  in order to construct a regular Zariski equisingular deformation of $(V,0)$ such that one of the fibers is the germ of a Nash hypersurface defined over $\ovl\Q$. By a refinement of a theorem of Varchenko \cite{Va} due to  A. Parusi\'nski and L. P\u{a}unescu \cite{PP} such a deformation is a Whitney equisingular deformation so it is topologically trivial and the trivialization is subanalytic and arc-analytic. Then we use the Artin-Mazur Theorem to transform our germ of a Nash set into the germ of an algebraic set (still defined over $\ovl\Q$) by a local diffeomorphism. For Theorem \ref{thm2} the idea is to apply essentially the same procedure to the graph of $g$, and the main difference concerns the part where we transform a Nash function germ into an algebraic function germ since Artin-Mazur Theorem is not sufficient to do this transformation. This part requires the construction of a particular deformation of the Nash set germ which is topologically trivial thanks to Thom-Mather Isotopy Lemma.\\
The paper is organized as follows: the first  and main part is devoted to give an algebraic statement concerning algebraic power series with complex coefficients solutions of algebraic equations with coefficients in $\ovl\Q$. It shows that such solutions are $\C$-points of a family of algebraic solutions defined over $\ovl\Q$ (see Theorem \ref{thm_rat}). In the next parts we apply this statement to prove Theorem \ref{thm1} and then Theorem \ref{thm2}, essentially by proving that the approach used in \cite{BPR} remains valid in our situation.

\begin{rem}
Let us mention that B. Teissier provided in \cite{Te} an example of the germ of a complex algebraic surface in $(\C^3,0)$ defined by a polynomial equation with coefficients in $\Q[\sqrt{5}]$ which is not Whitney equisingular to the germ of an algebraic set defined over $\Q$. So we cannot replace $\ovl\Q$ by $\Q$ in the statement of Theorem \ref{thm1}.
\end{rem}

\begin{rem}
It is known that the germ of an analytic set is not always diffeomorphic to the germ of an algebraic set (see \cite{Wh}). \\
Let us mention that in general the germ of an algebraic set is neither diffeomorphic to the germ of an algebraic set defined over $\ovl\Q$. For instance let us consider the germ of the curve $(V,0)\subset (\R^2,0)$  defined by the equation
$$xy(x-y)(x-\xi y)$$
where $\xi\in\R$ is a transcendental number.
Indeed $(V,0)$ is the union of four lines whose  cross-ratio is $\xi$. If $(V,0)$ were diffeomorphic to the germ of an algebraic set  $(W,0)$ defined over $\ovl\Q$, the differential of the diffeomorphism germ would induce a bijective linear map between the tangent spaces at 0 of $V$ and $W$. Such a linear map preserves the cross-ratio, so the tangent space of $(W,0)$ at 0 would be the union of four lines whose cross-ratio is equal to $\xi$ and this would not be possible since $(W,0)$ would be defined by algebraic equations with coefficients in $\ovl\Q$.\\
This example extends to the case where $\K=\C$ similarly as done in Example 6.2 \cite{BPR}.

\end{rem}

\begin{rem}
In general for a given analytic map germ $g:(\K^n,0)\lgw (\K^m,0)$ there is no germ of a homeomorphism $h:(\K^n,0)\lgw (\K^n,0)$ such that $g\circ h$ is the germ of a polynomial map: just take $g:(\K,0)\lgw (\K^2,0)$ given by $g(x)=(x,e^x)$ (see Example 6.3 \cite{BPR}). In particular Theorem \ref{thm2} cannot be extended to analytic map germs $(\K^n,0)\lgw (\K^m,0)$.\\
But in general, even if $g:(\K^n,0)\lgw (\K^m,0)$ is the germ of a polynomial map there is no germ of a homeomorphism $h:(\K^n,0)\lgw (\K^n,0)$ such that $g\circ h$ is the germ of a polynomial map defined over $\ovl\Q$. Indeed let $g:(\C,0)\lgw (\C^2,0)$ be defined by $g(x)=(x,\xi x)$ where $\xi\in\C$ is a transcendental number. If there were a homeomorphism germ $h:(\C,0)\lgw (\C,0)$ such that $g\circ h$ is the germ of a polynomial map defined over $\ovl\Q$ then both $h(x)$ and $\xi h(x)$  would  be algebraic over $\ovl\Q[x]$. But this would imply that $\xi$ is algebraic over $\Q$ which is not possible.
\end{rem}

\begin{rem}
In Theorem \ref{thm2} we do not know if the germ of a homeomorphism can be chosen to be arc-analytic or subanalytic. Indeed the proof of this result goes as follows: first we construct a Zariski equisingular deformation of the graphs of $g$ and of the function germs defining $(V,0)$ with the graphs of a Nash function germ $\wdt g$ and of function germs defining the germ of a Nash set $(\wdt V,0)$. Using Artin-Mazur Theorem we can reduce the situation to the case where $(\wdt V,0)$ is the germ of an algebraic set and $\wdt g$ is a unit $u$ times a polynomial function germ $ P$. Then, in \cite{BPR} is constructed a Thom stratification of the deformation
$$ (t,x)\lgw (1-t)u(0) P(x)+tu(x)P(x),$$
which shows (by Thom-Mather Isotopy Lemma) that the function germ $\wdt g=uP$ is homeomorphic to the function germ $P$. But Thom-Mather Isotopy Lemma does not provide an arc-analytic or subanalytic homeomorphism in general.
\end{rem}

\noindent\textbf{Notation and terminology} We will denote by $x$ and $y$ the vectors of indeterminates $(x_1,\ldots,x_n)$ and $(y_1,\ldots,y_m)$. The notation $x^i$ denotes the vector of indeterminates $(x_1,\ldots,x_i)$ for any $i\leq n$. When $\K=\R$ or $\C$, we denote by $\K\{x\}$ the ring of convergent power series with coefficients in $\K$, and by $\K\lg x\rg$ the ring of algebraic power series with coefficients in $\K$. This means that $\K\lg x\rg$ is the subring of $\K[[x]]$ whose elements are algebraic over $\K[x]$. We have that $\K\lg x\rg\subset \K\{x\}$, i.e. every algebraic power series is convergent.\\
\\
Let $\K=\C$ or $\mathbb{R}$. Let $\Omega$ be an open subset of
$\K^n$ and let $f$ be an analytic function on $\Omega$. We say that $f$ is a \emph{Nash
function} at $p\in\Omega$ if its Taylor expansion at $p$ is an algebraic power series.  An 
analytic function on $\Omega$ is a Nash function if it is a Nash function at every point
of $\Omega$. An analytic mapping $\varphi:\Omega\to\K^N$ is a \emph{Nash mapping} if all
its components are Nash functions on $\Omega$.\\
A subset $X$ of $\Omega$ is called a \emph{Nash subset} of $\Omega$ if for every
$p\in\Omega$ there exist an open neighborhood $U$ of $p$ in $\Omega$ and Nash
functions $f_1,\dots,f_s$ on $U$, such that $X\cap U=\{z\in U  \mid     f_1(z)=\cdots=f_s(z)=0\}$. A
germ $X_p$ of a set $X$ at $p\in\Omega$ is a \emph{Nash germ} if there exists an open
neighborhood $U$ of $p$ in $\Omega$ such that $X\cap U$ is a Nash subset of $U$. \\
A Nash function germ  is said to be defined over $\ovl\Q\cap \K$ if it satisfies a nontrivial polynomial equation with coefficients in $\ovl\Q\cap\K$. This is equivalent to say that its Taylor expansion at a $\ovl\Q\cap\K$-point is an algebraic power series whose coefficients are in $\ovl\Q\cap\K$, i.e. an element of $(\ovl\Q\cap\K)[[x]]$. A Nash set is said to be defined over $\ovl\Q\cap \K$ if it is locally defined by Nash function germs defined over $\ovl\Q\cap\K$.

\noindent\textbf{Acknowledgements.} This work originates from discussions with Adam Parusi\'nski. I wish to thank him warmly  for the many fruitful discussions we had on this problem and  his helpful remarks and valuable suggestions concerning the earlier versions of this paper. I also thank the referee for their useful and suitable comments and remarks.


\section{An approximation result}
We begin by stating the main result of this part:
\begin{thm}\label{thm_rat}
Let $f(x,y)\in \ovl\Q \lg x\rg[y]^p$ and let us consider a solution $y(x)\in\C \lg x\rg^m$ of $$f(x,y(x))=0.$$
Then there exist a new set of indeterminates $t=(t_1,\ldots,t_r)$, a vector of algebraic power series 
$$y(t,x)=\sum_{\a\in\N^n}y_\a(t)x^\a\in\ovl\Q\lg t,x\rg^m$$
and $\tt=(\tt_1,\ldots,\tt_r)\in\C^r$ belonging to the domain of convergence of all the $y_\a(t)$ such that
$$y(x)=y(\tt,x)
\text{ and } f(x,y(t,x))=0.$$

\end{thm}

\begin{rem}
This theorem is not true if we replace $\ovl\Q$ by $\Q$. For instance let $x$ and $y$ be single indeterminates and set $f=y^2-2x^2$. Then there is no algebraic power series $y(x,t)\in\Q\lg x,t\rg$ such that 
$$y(x,t)^2-2x^2=0$$
but
we have
$$f(x,\sqrt{2}x)=0.$$
\end{rem}

\begin{proof}[Proof of Theorem \ref{thm_rat}]
If $y(x)\in \ovl\Q\lg x\rg^m$ then we take $r=0$ and there is nothing to prove.\\
Let us assume that $y(x)\in\C \lg x\rg^m\backslash \ovl\Q\lg x\rg^m$.
By Lemma \ref{lem_alg} given below we may assume that there exist $y'(t,u,v,x )\in\ovl\Q\lg t,u,v,x \rg^m\cap\ovl\Q[t,u,v][[x]]^m$, where $t=(t_1,\ldots, t_{r})$ and $u$ and $v$ are single indeterminates, and $\tt\in\C^{r}$, $\uu\in\C$, $\vv\in\C$ such that
\begin{equation}y(x )=y'(\tt,\uu,\vv,x ).\end{equation}
Moreover we may assume that $\tt_1$,\ldots, $\tt_{r}$ are algebraically independent over $\ovl\Q$, $\uu=\frac{1}{R(\tt_1,\ldots,\tt_{r})}$ for some polynomial $R\in\ovl\Q[t_1,\ldots,t_r]$ such that $R(\tt_1,\ldots,\tt_r)\neq 0$, and $\vv$ is finite over $\L:=\ovl\Q(\tt_1,\ldots,\tt_{r})$.\\
Let $P(t_1,\ldots,t_{r},v)\in\ovl\Q(t)[v]$ be the monic polynomial of minimal degree in $v$ such that
$$P(\tt_1,\ldots,\tt_{r},\vv)=0.$$
Let $D\subset \C^{r}$ be  the discriminant locus of $P(t,v)$ seen as a polynomial in $v$ (i.e. $D$ is the locus of points $q\in\C^{r}$ such that $q$ is a pole of one of the coefficients of $P$ or such that $P(q,v)$ has at least one multiple root). Since $P(\tt_1,\ldots,\tt_{r},v)$ has no multiple roots in  an algebraic closure of $\L$, the point $\tt$ is not in $D$.  Then there exist   $\mathcal U\subset \C^{r}\backslash D$  a simply connected open neighborhood of $p$ and analytic functions
$$w_i : \mathcal U\lgw \C, \ \ i=1,\ldots, d$$
such that
$$P(t,v)=\prod_{i=1}^d(v-w_i(t))$$
and $w_1(\tt_1,\ldots,\tt_{r})=\vv$.\\
Moreover  the $t\lgm w_i(t)$ are algebraic functions over $\ovl\Q[t]$. In particular the Taylor series of $w_1$ at a point of $\mathcal U\cap\ovl\Q^{r}$ is an algebraic power series with algebraic coefficients.\\
Since the polynomial $R$ is not vanishing at $p$  the function 
$$t\in\C^{r}\backslash\{R=0\}\lgm \frac{1}{R(t)}$$ is also an analytic function which is algebraic  over $\ovl\Q[t]$ and so its Taylor series at a point of $\ovl\Q^{r}\backslash\{R=0\}$ is an algebraic power series with algebraic coefficients.\\
Let $q:=(q_1,\ldots,q_{r})\in \ovl\Q^{r}\cap\mathcal U\backslash\{R=0\}$ such that $\tt$ belongs to an open polydisc  $\Delta$ centered at $q$ and such that $\Delta\subset \mathcal U\backslash\{R=0\}$.  We denote by $\phi_1(t)$ and $\phi_2(t)\in\ovl\Q\lg t\rg$ the Taylor series of  $t\lgm \frac{1}{R(t)}$  and $w_1$ at $q$. For simplicity we can make a translation and assume that $q$ is the origin of $\C^{r}$. In particular the series $\phi_1(t)$ and $\phi_2(t)$ are convergent at $\tt$.\\
We have that
$$f(x,y'(\tt_1,\ldots,\tt_{r},\uu,\vv,x ))=0$$
or equivalently
$$f(x,y'(\tt_1,\ldots,\tt_{r},\phi_1(\tt_1,\ldots,\tt_{r}),\phi_2(\tt_1,\ldots,\tt_{r}),x ))=0.$$
The function
$$(t,x )\lgm F(t,x):=f(x,y'(t_1,\ldots,t_{r},\phi_1(t_1,\ldots,t_{r}),\phi_2(t_1,\ldots,t_{r}),x ))$$ is an algebraic function over $\ovl \Q[t,x ]$.  So if $F(t,x)\not\equiv 0$ there exists an algebraic function $(t,x)\lgm g(t,x)$ such that
$$(t,x )\lgm g(t,x)F(t,x)$$
is a nonzero polynomial function. Indeed if 
$$a_0(t,x)T^e+a_1(t,x)T^{e-1}+\cdots+a_e(t,x)$$
is a polynomial of minimal degree having $F(t,x)$ as a root then $a_e(t,x)\not\equiv 0$ and we can choose
$$g(t,x):=-a_0(t,x) F(t,x)^{e-1}-a_1(t,x)F(t,x)^{e-2}+\cdots-a_{e-1}(t,x)$$
so we have that
$$g(t,x)F(t,x)=a_e(t,x).$$
Since   $F(\tt,x)=0$ we have that $a_e(\tt,x)=0$ but $\tt_1$, \ldots, $\tt_{r}$, $x$ being algebraically independent over $\ovl\Q$ we obtain that $a_e(t,x)\equiv 0$ which is a contradiction. Thus  we have that
$$F(t,x)=f\big(x,y'(t_1,\ldots,t_{r},\phi_1(t_1,\ldots,t_{r}),\phi_2(t_1,\ldots,t_{r}),x )\big)=0.$$
This proves the theorem by defining
$$y(t,x)=y'(t,\phi_1(t),\phi_2(t),x).$$
Since $\phi_1(t)$ and $\phi_2(t)$ are convergent power series at $\tt$ and since $y'(t,u,v,x )\in\ovl\Q\lg t,u,v,x \rg^m\cap\ovl\Q[t,u,v][[x]]^m$ then all the series $y_\a(t)$ are convergent at $\tt$.

\end{proof}

\begin{rem}
Let us assume that $f(x,y)\in\Q\lg x\rg[y]^p$. In the proof of Theorem \ref{thm_rat} let us assume that $r=0$, i.e. the coefficients of $y(x)$ belong to a finite field extension of $\Q$. In this case the analytic function $w_1$ is a constant function whose value is in $\ovl\Q\backslash \Q$. This is why  we need to work with the algebraically closed field $\ovl\Q$ and not only with $\Q$.

\end{rem}

\begin{lem}\label{lem_alg}
Let $f\in\C\langle x\rangle^m\backslash\ovl\Q\lg x\rg^m$. Then there exist complex numbers $\tt_1$, \ldots, $\tt_r$, $\uu$ and $\vv$ with $r\geq 1$ and  $F\in\ovl\Q\lg t,u,v,x\rg^m$, where $t=(t_1,\ldots, t_r)$ and $u$ and $v$ are single indeterminates, such that
\begin{itemize}
\item $F\in\ovl\Q[t,u,v][[x]]^m$,
\item $f(x)=F(\tt_1,\ldots,\tt_r,\uu,\vv,x),$
\item the extension $\ovl\Q\lgw \ovl\Q(\tt_1,\ldots,\tt_{r})$ is purely transcendental, 
\item $\uu=\frac{1}{R(\tt_1,\ldots,\tt_{r})}$ for some polynomial $R\in\ovl\Q[t_1,\ldots,t_r]$ with $R(\tt_1,\ldots,\tt_r)\neq 0$,
\item  $\vv$ is finite over $\ovl\Q(\tt_1,\ldots,\tt_r)$.
\end{itemize}

\end{lem}

\begin{proof}
Let $\K$ be the field extension of $\ovl\Q$ generated by the coefficients of the minimal polynomials of the components of $f$. Then the coefficients of the components of $f$ belong to a finite field extension of $\K$ (see for instance \cite{CK}). Let us replace $\K$ by this finite field extension. There exists a purely transcendental finitely generated field extension $\ovl \Q\lgw \L$ such that $\L\lgw\K$ is finite. By enlarging $\K$ we may assume that $\L\lgw \K$ is normal. By the primitive element theorem $\K=\L(a)$ for some $a\in\C$ algebraic over $\L$.\\
Let us write $f(x)=(f_1(x),\ldots,f_m(x))$. We can write, for $i=1,\ldots, m$:
$$f_i(x)=\sum_{k=0}^{d-1}a^kf_{i,k}(x)$$ 
where $d$ is the degree of $a$ over $\L$ and the $f_{i,k}(x)$ are power series with coefficients in $\L$. Let us denote by 
$$a_1=a,\ a_2,\ldots, a_d$$
the conjugates of $a$ over $\L$. Since $f_i(x)$ is algebraic over $\K[x]$ and $\L\lgw \K$ is an algebraic extension, we have that $f_i(x)$ is algebraic over $\L[x]$. Let $P_i(x,y)=\sum_{\a,l}p_{\a,l}x^\a y^l\in\L[x,y]$ be a nonzero vanishing polynomial of $f_i(x)$ and let $\si$ be a $\L$-automorphism of $\K$ such that $\si(a)=a_j$ for some $j$. It induces a $\L[[x]]$-automorphism of $\K[[x]]$ defined by $\si(\sum_\a c_\a x^\a)=\sum_\a\si(c_\a)x^\a$. Then we have that
$$0=\si(P_i(x,f_i(x))=\sum_{\a,l}\si(p_{\a,l})x^\a\si(f_i(x))^l)=\sum_{\a,l}\si(p_{\a,l})x^\a\left(\sum_{k=0}^{d-1}a_j^kf_{i,k}(x)\right)^l.$$
Thus for every $i=1,\ldots, m$ and $j=1,\ldots, d$  the power series
$$\sum_{k=0}^{d-1}a_j^kf_{i,k}(x)\in\K[[x]]$$ 
is algebraic over $\K[x]$. Let $M$ be the (non-singular) $d\times d$ Vandermonde matrix associated to the $a_j$. Then we have that
$$\wdt f_i(x)=M \ovl f_i(x)$$
where $\wdt f_i(x)$ is the vector whose entries are the $\sum_{k=0}^{d-1}a_j^kf_{i,k}(x)$, for $j=1$, \ldots, $d$, and $\ovl f_i(x)$ is the vector whose entries are the $f_{i,k}(x)$. Then $\ovl f_i(x)=M^{-1}\wdt f_i(x)$, thus the $f_{i,k}(x)$ are algebraic over $\K[x]$ and so over $\L[x]$. This shows that $f_{i,k}(x)\in\L\lg x\rg$ for every $i$ and $k$.\\

Let $\tt_1$, \ldots, $\tt_{r}$ be a  transcendence basis of $\L/\ovl\Q$. Then by Lemma \ref{eis} given below we have that:
$$f_{i,k}=\sum_{\a\in\N^n}\frac{S_{i,k,\a}(\tt_1,\ldots, \tt_{r })}{R_{i,k}(\tt_1,\ldots,\tt_{r })^{|\a|}}x^\a$$
for some polynomials $S_{i,k,\a}$ and $R_{i,k}\in\ovl\Q[t_1,\ldots, t_r]$.
By replacing each $R_{i,k}$ by $\prod_{j,l}R_{j,l}$ and multiplying every $S_{i,k,\a}$ by $\prod_{(j,l)\neq(i,k)}^{d-1}R_{j,l}^{|\a|}$ we may assume that $R_{i,k}=R_{i',k'}=R$ for every $(i,k)$ and $(i',k')$. \\
The power series
$$f^*_{i,k}(x)=f_{i,k}(R(\tt_1,\ldots,\tt_{r })x_1,\ldots, R(\tt_1,\ldots,\tt_{r })x_n)=\sum_{\a\in\N^n}S_{i,k,\a}(\tt_1,\ldots, \tt_{r }) x^\a$$
belongs to $\ovl\Q(\tt_1,\ldots,\tt_{r })\lg x\rg$ since $f_{i,k}(x)\in \ovl\Q(\tt_1,\ldots,\tt_{r })\lg x\rg$.\\
Thus we have that
$$f_{i,k}^*=F_{i,k}(\tt_1,\ldots,\tt_{r },x)$$
with 
$$F_{i,k}:=\sum_{\a\in\N^n}S_{k,\a}(t_1,\ldots, t_{r })x^\a\in\ovl\Q[t_1,\ldots, t_{r }][[x]]$$
for every $i$ and $k$  where the $t_i$ are new indeterminates.\\
 Moreover let $P_{i,k}(t_1,\ldots, t_{r },x,y)\in\ovl\Q[t_1,\ldots, t_{r },x,y]$, where $y$ is a new indeterminate, be a nonzero polynomial with $P_{i,k}(\tt,x,f_{i,k}^*(x))=0$. Since $F_{i,k}\in\ovl\Q[t_1,\ldots,t_{r }][[x]]$ for every $k$, we can write
 $$P_{i,k}(t_1,\ldots,t_{r },x,F_{i,k}(t_1,\ldots,t_{r },x))=\sum_{\b\in\N^n} P_{i,k,l}(t_1,\ldots,t_{r })x^\b$$
 for some polynomials $P_{i,k,\b}\in\ovl\Q[t_1,\ldots,t_{r }]$. Thus $P_{i,k,\b}(\tt_1,\ldots, \tt_{r })=0$ for every $i$, $k$ and $\b$, but since $\tt_1$, \ldots, $\tt_{r }$ are algebraically independent over $\ovl\Q$ we have that
 $$P_{i,k,\b}(t_1,\ldots, t_{r })=0$$ for every $i$, $k$ and $\b$ so
 $$P_{i,k}(t_1,\ldots,t_{r },x,F_{i,k}(t_1,\ldots,t_{r },x))=0$$
 and this implies that $F_{i,k}\in\ovl\Q\lg t_1,\ldots, t_{r },x\rg$. In particular if $u$ denotes a new indeterminate we have that
 $$F_{i,k}(t_1,\ldots,t_{r },ux_1,\ldots,ux_n)\in\ovl\Q\lg t_1,\ldots,t_r,u,x\rg\cap\ovl\Q[t_1,\ldots,t_r,u][[x]]\ \ \forall k.$$
 Finally we set $t=(t_1,\ldots,t_r)$ and 
 $$F_i(t,x)=\sum_{k=0}^{d-1}v^kF_{i,k}(t_1,\ldots,t_{r },ux_1,\ldots,ux_n)$$
 where $v$ denotes a new indeterminate. Thus
the result is proven  with $F$ the vector whose components are the $F_i$ and 
 $\uu=\frac{1}{R(\tt_1,\ldots,\tt_{r })}$ and $\vv=a$.
 \end{proof}
The  following version of Eisenstein Lemma is essentially  Lemma 2.2 \cite{To2} and the proof is the same - but we give it here for the convenience of the reader:
\begin{lem}[Eisenstein Lemma]\label{eis}
Let $f\in\ovl \Q(\tt_1,\ldots, \tt_r)\lg x\rg$ be an algebraic power series where the $\tt_i\in\C$ are algebraically independent over $\ovl \Q$. Then there exist a polynomial $R(t)\in\ovl\Q[t]$ and polynomials $S_\a(t)\in \ovl\Q[t]$ for every $\a\in\N^n$, where $t=(t_1,\ldots, t_r)$ is a vector of new indeterminates, such that
$$f(x)=\sum_{\a\in\N^n}\frac{S_\a(\tt_1,\ldots, \tt_r)}{R(\tt_1,\ldots,\tt_r)^{|\a|}}x^\a.$$

\end{lem}
\begin{proof}[Proof of Lemma \ref{eis}]
Let $P(x,y)\in\ovl\Q(\tt_1,\ldots, \tt_r)[x,y]$ be a minimal polynomial of $f$, i.e. a generator of the kernel  of the ring morphism :
$$\ovl\Q(\tt_1,\ldots,\tt_r)[x,y]\lgw \ovl\Q(\tt_1,\ldots, \tt_r)[[x]]$$
$$p(x,y)\lgm p(x,f(x))$$ 
Let us set
$$e:=\ord_x\left(\frac{\partial P}{\partial y}(x,f(x))\right).$$
We have that $e<\infty$ since $P(x,y)$ is a minimal polynomial of $f(x)$. \\
Let us write $f=\sum_{\a\in\N^n}f_\a(\tt)x^\a$ where $\tt=(\tt_1,\ldots,\tt_r)$ and $f_\a(\tt)\in\ovl\Q(\tt)$. Let $b(\tt)\in\ovl\Q[\tt]$ be a common denominator of the $f_\a(\tt)$ for $|\a|\leq 2e+1$. We have that Lemma \ref{eis} is satisfied by $f$ if and only if it is satisfied by $b(\tt)f$. Thus we may replace $f$ by $b(\tt)f$. In this case a minimal polynomial of $b(\tt)f$ is
$$P'(x,y):=b(\tt)^{\deg_y(P)}P\left(x,\frac{y}{b(\tt)}\right).$$
Moreover by multiplying $P'(x,y)$ by an element of $\ovl\Q(\tt_1,\ldots,\tt_r)$ we may assume that $P'(x,y)\in\ovl\Q[\tt_1,\ldots,\tt_r][x,y]$. Then we have
$$e=\ord_x\left(\frac{\partial P'}{\partial y}(x,b(\tt)f(x))\right).$$
Thus we may replace $f$ by $b(\tt)f$ and assume that $f_\a(\tt)\in\ovl\Q[\tt]$ for $|\a|\leq 2e+1$.\\
\\
We define
$$P^*(u,x,y):=P(ux_1,\ldots,ux_n,y)\in\ovl\Q[\tt_1,\ldots,\tt_r,x_1,\ldots,x_n][u,y]$$
and
$$f^*(u,x):=f(ux_1,\ldots,ux_n)$$
where $u$ is a new indeterminate. Then $P^*(u,x,f^*(u,x))=0$ so $f^*\in\ovl\Q(t,x)\lg u\rg$.\\
Let us denote by $ {f^*}^{(2e+1)}(u)$ the $(2e+1)$-truncation of $f^*(u)$ (i.e. we remove from $f^*(u)$ all the monomials which are divisible by $u^{2e+2}$). Then we have that
$$P^*(u,x,{f^*}^{(2e+1)})\in (u)^{2e+2}\ \text{ and }  \frac{\partial P^*}{\partial y}(u,x,{f^*}^{(2e+1)})\in (u)^e\backslash(u)^{e+1}.$$
Let us set
$$y=u^{e+1}y'+{f^*}^{(2e+1)}$$
where $y'$ is a new indeterminate. Then we have that
$$P^*(u,x,y)=P^*(u,x,{f^*}^{(2e+1)})+\frac{\partial P^*}{\partial y}(u,x,{f^*}^{(2e+1)})u^{e+1}y'+u^{2e+2}{y'}^2Q(u,y')$$
for some polynomial $Q$.
Thus the equation $P^*(u,x,y)=0$ is equivalent to
\begin{equation}\label{eq1}\frac{P^*(u,x,{f^*}^{(2e+1)})}{u^{2e+1}}+\frac{\frac{\partial P^*}{\partial y}(u,x,{f^*}^{(2e+1)})}{u^e}y'+u{y'}^2Q(u,y')=0.\end{equation}
Since $f_\a(\tt)\in\ovl\Q[\tt]$ for $|\a|\leq 2e+1$ we have that ${f^*}^{(2e+1)}\in\ovl\Q[\tt,x,u]$. Since  the coefficients in \eqref{eq1} are polynomials in $u$, $x$ and the $f_\a(\tt)$ for $|\a|\leq 2e+1$, they belong to $\ovl\Q[\tt,x,u]$.
Let  $R(\tt,x)\in\ovl\Q[\tt,x]$ be defined by 
$$R(\tt,x)=\left(\frac{\frac{\partial P^*}{\partial y}(u,x,{f^*}^{(2e+1)})}{u^e}\right)_{|u=0}.$$
Since $\ord_u\left(\frac{P^*(u,x,{f^*}^{(2e+1)})}{u^{2e+1}}\right)\geq 1$ we see that $R(\tt,x)^2$ divides 
$$\dfrac{P^*(R(\tt ,x)^2u',x,{f^*}^{(2e+1)}(R(\tt ,x)^2u'))}{(R(\tt ,x)^2u')^{2e+1}}$$ where $u'$ is a new indeterminate. \\
Thus by replacing $y'$ by $R(\tt,x)y''$ and $u$ by $R(\tt,x)^2u'$ in Equation \eqref{eq1}, and dividing it by $R(\tt,x)^2$ we have that \eqref{eq1} is equivalent to
\begin{equation}\label{eq2}A_1(\tt ,x,u' )+(1+u' A_2(\tt ,x,u' ))y''+u' {y''}^2A_3(\tt ,x,u' )=0\end{equation}
where the $A_i$ belongs to $\ovl\Q[\tt,x,u' ]$.  
By the implicit function theorem (or Hensel's Lemma) this equation has a unique solution in $\ovl\Q\lg\tt ,x, u' \rg$ which is necessarily $g^*:=\frac{f^*(R(\tt ,x)^2u' )}{R(\tt ,x)}$. Moreover we can also apply Hensel's Lemma to this equation to see that it has a unique solution in the completion of $\ovl\Q[\tt,x,u']$ with respect to the ideal generated by $u'$, i.e. in the ring $\ovl\Q[\tt,x][[u']]$. Thus
$$g^*\in\ovl\Q[\tt,x][[u']]\cap\ovl\Q\lg \tt,x, u' \rg.$$
  In particular the coefficients $g^*_k$ defined by $g^*(u' )=\sum_{k\geq 0}g^*_k{u' }^k$ are polynomials over $\ovl\Q$ depending on the $\tt_i$ and the $x_j$. Moreover we have that
$$f^*(u)=\sum_{k\geq 0}\frac{g^*_k(\tt ,x)}{R(\tt ,x)^{2k-1}}u^k.$$
On the other hand we have that
$$f^*(u)=\sum_{k\geq 0}\left(\sum_{|\a|=k}f_\a(\tt ) x^\a\right)u^k$$
hence
\begin{equation}\label{eq3}\sum_{|\a|=k}f_\a(\tt ) x^\a=\frac{g^*_k(\tt ,x)}{R(\tt ,x)^{2k-1}} \ \ \forall k\in\N.\end{equation}
For every  $\a\in\N^n$, let us write $f_\a(\tt )=\frac{h_\a(\tt )}{l_{|\a|}(\tt )}$ where $h_\a(\tt )$, $l_{|\a|}(\tt )\in\ovl\Q[\tt]$ and $l_{|\a|}(\tt )$ is coprime with $\sum_{|\a|=k}h_\a(\tt ) x^\a$. Then $l_{|\a|}(\tt )$ divides $R(\tt ,x)^{2|\a|-1}$. Let $r(\tt )$ be the greatest divisor of $R(\tt ,x)$ belonging to $\ovl\Q[\tt]$. Then there exists $d_{|\a|}(\tt )\in\ovl\Q[\tt]$ such that $l_{|\a|}(\tt )d_{|\a|}(\tt )=r(\tt )^{2|\a|-1}$. Thus we have that 
$$f_\a(\tt )=\frac{h_\a(\tt )d_{|\a|}(\tt )}{r(\tt )^{2|\a|-1}}.$$
This proves the lemma.
\end{proof}

Theorem \ref{thm_rat} allows us to prove the  following version of the Nested Artin-P\l oski-Popescu Approximation Theorem:

\begin{thm}
\label{nest_ploski} 
Let $f(x,y)\in \ovl\Q \lg x\rg[y]^p$ and let us consider a solution $y(x)\in\C \{x\}^m$ of $$f(x,y(x))=0.$$
 Let us assume that $y_i(x)$ depends only on $(x_1,\ldots, x_{\si(i)})$ where $i\lgm \si(i)$ is an increasing function. Then there exist two sets of indeterminates  $z=(z_1,\ldots, z_s)$ and $t=(t_1,\ldots,t_r)$, an increasing function $\t$, convergent power series $z_i(x)\in\C \{x\}$ vanishing at 0 such that $z_1(x)$, \ldots, $z_{\t(i)}(x)$ depend only on $(x_1,\ldots, x_{\si(i)})$, complex numbers $\tt_1$, \ldots, $\tt_r\in\C$ and  an  algebraic power series vector solution $y(t,x,z)\in \ovl\Q\lg t,x,z\rg^m$ of
 $$f(x,y(t,x,z))=0,$$ such that for every $i$, 
 $$y_i(t,x,z)\in\ovl \Q \lg  t,x_1,\ldots,x_{\si(i)},z_1,\ldots,z_{\t(i)}\rg,$$
 $y(\tt,x,z)$ is well defined and
 $y(x)=y(\tt,x,z(x)).$
\end{thm} 
\begin{proof}
By Theorem 1.2 \cite{BPR} there exist a new set  of indeterminates $z=(z_1,\ldots, z_s)$, an increasing function $\t$, convergent power series $z_i(x)\in\C \{x\}$ vanishing at 0 such that $z_1(x)$, \ldots, $z_{\t(i)}(x)$ depend only on $(x_1,\ldots, x_{\si(i)})$, and  an  algebraic power series vector solution $y(x,z)\in \C\lg x,z\rg^m$ of
 $$f(x,y(x,z))=0,$$ such that for every $i$, 
 $$y_i(x,z)\in\C \lg  x_1,\ldots,x_{\si(i)},z_1,\ldots,z_{\t(i)}\rg,$$
 and
 $y(x)=y(x,z(x)).$\\
   Then we apply Theorem \ref{thm_rat} to the vector $y(x,z)$.
\end{proof}


\section{Proof of Theorem \ref{thm1}}

The proof is similar to the proof of Theorem 1.2 \cite{BPR} and so we will refer several times to this paper for details. For convenience $x^{n-1}$ will denote the vector of indeterminates $(x_1,\ldots,x_{n-1})$ and, more generally, $x^i$ will denote the vector of indeterminates $(x_1,\ldots,x_i)$.\\
Firstly we consider the case $\K=\C$. Let $g_1$, \ldots, $g_k\in\C\{x\}$ be the defining equations of $(V,0)$. By a linear change of coordinates we may assume that the $g_i$ are Weierstrass polynomials in $x_n$:
 \begin{align*}
g_{s} ( x)= x_n^{r_s}+ \sum_{j=1}^{r_s} a_{n-1,s,j} (x^{n-1}) x_n^{r_s-j} \ \ \forall s=1,\ldots, k
\end{align*}
and 
\begin{equation}\label{mult} \text{mult}_0(g_s)=r_s \ \ \forall s=1,\ldots, k.\end{equation}
Then the $a_{n-1,s,j}$ are arranged in a row vector $a_{n-1}\in\C\{x^{n-1}\}^{p_n}$ with $p_n=\sum_sr_s$.  Let $f_n$ denote the product of the $g_s$. Let $\D_{n,i}$ denote the $i$-th generalized discriminant of $f_n$ seen as a polynomial in $x_n$  (see 4.2 \cite{BPR}). This is a polynomial depending on $a_{n-1}$. Then let $\D_{n,j_n}(a_{n-1})$ be the first non-vanishing generalized discriminant.\\
After a linear change of coordinates in $x_1$, \ldots, $x_{n-1}$ we may assume, by the Weierstrass Preparation Theorem,  that:
$$\D_{n,j_n}(a_{n-1}) =u_{n-1} (x^{n-1}) (x_{n-1}^{p_{n-1}}
 + \sum_{j=1}^{p_{n-1}} a_{n-2,j} (x^{n-2}) x_{n-1}^{p_{n-1}-j} ),$$
where $u_{n-1}(0)\neq 0$ and for all $j$, $a_{n-2,j}(0)=0$. \\
We carry on with this construction (exactly 
as in 4.2 \cite{BPR}) and define a sequence of Weierstrass polynomials $f_i(x^i)$ for $i=1$, \ldots, $n-1$ such that
 $f_i= x_i^{p_i}+ \sum_{j=1}^{p_i} a_{i-1,j} (x^{i-1}) x_i^{p_i-j}  $
 is the Weierstrass polynomial associated to the first non identically zero generalized discriminant $\Delta_{i+1,j_{i+1}} ( a_{i} )$ of $f_{i+1}$, 
where   $a_{i}$ denotes the vector $(a_{i,1} , \ldots , a_{i,p_{i+1}} )$:
  \begin{align}\label{polyf_i}
 \Delta_{i+1,j_{i+1}} ( a_{i} ) =  u_{i} (x^{i})  (x_i^{p_i}+ \sum_{j=1}^{p_i} a_{i-1,j} (x^{i-1}) x_i^{p_i-j}  ) ,  
 \quad i=0,\ldots,n-1.  
\end{align}
 Thus the vector  
of power series $a_i$ satisfies  
  \begin{align}\label{disci}
\Delta_{i+1,k} ( a_{i} )\equiv 0 \qquad k<j_{i+1}   ,  \quad i=0,\ldots ,n-1. 
\end{align}
In particular  $\Delta_{1,j_1}(a_0)$ is a constant.\\
Then we use Theorem \ref{nest_ploski} to see that there exist two sets of indeterminates  $z=(z_1,\ldots, z_s)$ and $t=(t_1,\ldots, t_r)$, an increasing function $\tau$,  convergent power series $z_i(x)\in\C \{x\}$ {vanishing at 0}, complex numbers $\tt_1$, \ldots, $\tt_r\in\C$, algebraic power series $u_i (t,x^i,z)  \in \ovl\Q\lg t,x^{i},z_1,\ldots,z_{\tau(i)}\rg$ and  vectors of algebraic power series 

$$a_i(t,x^i,z) \in\ovl\Q\lg t,x^{i},z_1,\ldots,z_{\tau(i)}\rg^{p_i}$$
 such that the following holds:
\begin{enumerate}

\item[a)]  $z_1(x),\ldots, z_{\tau(i)}(x)$ depend only on $(x_1,\ldots, x_{i})$,
\item[b)] $a_i(t,x^i,z),\ u_i (t,x^i,z)$ are solutions of \eqref{polyf_i} and
\eqref{disci} 
\item[c)] $a_i(x^i)= a_i(\tt,x^i,z(x^i)),\   u_i(x^i)= u_i(\tt,x^i,z(x^i))$.\\
\end{enumerate}
Let $\ovl u_i$ be the constant coefficient of $u_i(x^i)$. Because $z(0)=0$ we have that $\ovl u_i=u_i(\tt,0,0)$ and $u_i(t,0,0)\in\ovl\Q\langle t\rangle$. In particular $u_i(\tt,0,0)\neq 0$.  Let $\g : \mathcal U\lgw \C^r$ be the analytic map defined by
$$\g(\la)=(1-\la)\q+\la \tt$$ where $\mathcal U$ is an  open connected neighborhood of the closed unit disc in $\C$ and $\q\in\ovl\Q^r$. Because $u_i(\tt,0,0)\neq 0$ and $\ovl\Q$ is dense in $\C$ we may choose $\q$ close enough to $\tt$ 
 such that 
$$u_i(\g(\la),0,0)\neq 0 \ \forall i, \ \forall \la\in \mathcal U.$$
Again because $\ovl\Q$ is dense in $\C$ we can find  $\q\in\ovl\Q$ close enough to $\tt$  such that the following are, for all $\la\in \mathcal U$, well defined convergent power series in $x$:
\begin{align*}
 F_n(\la,x): = \prod_s G_s(\la,x)  , \quad G_{s} (\la, x):=  x_n^{r_s}+ 
 \sum_{j=1}^{r_s} a_{n-1,s,j} (\g(\la),x^{n-1}, \la z(x^{n-1})) x_n^{r_s-j}  ,
\end{align*}
 \begin{align*}
 F_i(\la,x) := x_i^{p_i}+ \sum_{j=1}^{p_i} a_{i-1,j} (\g(\la),x^{i-1}, \la z(x^{i-1}) ) x_i^{p_i-j} , \quad i=1, \ldots , n-1,
\end{align*}
 \begin{align*} 
u_i(\g(\la),x^i,z(x^i)), \quad i=1, \ldots , n-1.\end{align*}
Finally we set $F_0\equiv 1$.  
Because $u_i(\g(\la),0,z(0))\neq 0$,  the family $F_{i} (\la, x)$ satisfies the assumptions of Theorem 3.3 \cite{PP} with $|\la|\leq 1$, i.e. the family is Zariski equisingular. Moreover by \eqref{mult} we have that
$$\text{mult}_0(G_s)=r_s \ \ \forall s=1,\ldots, k$$
so the familly is Zariski equisingular with transverse projections (see \cite{PP} Definition 4.1). So by Theorem 4.3 \cite{PP} this family is a regular Zariski equisingular family and so by Theorem  7.1 \cite{PP} it is Whitney equisingular. Thus $\{F_n(0,x)=0\}$ and $\{F_n(1,x)=0\}=\{f_n(x)=0\}$ are homeomorphic and the homeomorphism between them can be chosen to be subanalytic and arc-analytic. We have that
$F_n(0,x)\in\ovl\Q\langle x\rangle$ thus, by Theorem 3.2 \cite{BPR},  we may assume that $(V,0)$ is the germ of a Nash set defined over $\ovl\Q$. \\
When $\K=\R$ we may also assume that $(V,0)$ is the germ of a Nash set defined over $\ovl\Q\cap\K$. This follows  from the complex case by the same argument used in the proof of Corollary 4.1 \cite{BPR}.\\
Then we conclude with the following theorem:

\begin{thm}\label{AM}
Let $(V,0)\subset (\K^n,0)$ be a Nash set germ defined over $\ovl\Q\cap\K$. Then there exists a local Nash diffeomorphism $h:(\K^n,0)\rightarrow (\K^n,0)$ such that $h(V)$ is the germ of an algebraic subset of $\K^n$ defined over $\ovl\Q\cap\K$.

\end{thm}

\begin{proof} This follows from  Proposition \ref{AM2} given below which is a slight modification of Proposition 2 \cite{BK}. Indeed let $f: U\lgw \K^m$ be a Nash function such that $f^{-1}(0)=V$. Then by Proposition \ref{AM2} we have that $V=s^{-1}(\phi^{-1}(0))$. But $s: U\lgw s(U)$ is a Nash diffeomorphism by Prop. \ref{AM2} ii). So we set $h=s$ and $h(V)$ is an algebraic set equal to $\phi^{-1}(0)$, again by using the notations of Prop. \ref{AM2}.
\end{proof}

\begin{prop}\label{AM2}
Let $f: U\lgw \K^m$ be a Nash map defined on an open connected set $U\subset \K^n$ by algebraic power series with coefficients in $\ovl\Q\cap\K$. Then there exist an algebraic set $X\subset \K^n\times\K^N$, a polynomial map $\phi:X\lgw \K^m$ and a Nash map $s:U\lgw \K^n\times\K^N$ satisfying the following properties:
\begin{enumerate}
\item[i)] $s(U) \subset \text{Reg}(X)$ is a connected component of $p^{-1}(U)\cap X$, where $p:\K^n\times\K^N\lgw \K^n$ is the first projection,
\item[ii)] $p\circ s=\text{Id}_U$,
\item[iii)] $f=\phi\circ s$,
\item[iv)] the coefficients of the polynomials defining $X$ and $\phi$ are in $\ovl\Q\cap\K$.
\end{enumerate}

\end{prop}

\begin{proof}
The existence of $X$, $\phi$ and $s$ satisfying $i)$, $ii)$ and $iii)$ are given by Proposition 2 \cite{BK} in the general case where $f$ is defined by algebraic power series with coefficients in $\K$. In fact $X$ is the normalization of the Zariski closure of the graph of $f$ and $\phi$ is the restriction to $X$ of a generic linear map $\K^{n+N}\lgw \K^m$. In particular, since $f$ is assumed to be defined over $\ovl\Q\cap\K$, we have that $X$ is defined by polynomial equations with coefficients in $\ovl\Q\cap\K$. Because $\phi$ is generic we can choose such a $\phi$ with coefficients in $\ovl\Q\cap\K$ since this field is dense in $\K$. 
 
\end{proof}



\section{Proof of Theorem \ref{thm2}}

The proof is similar to the proof of Theorem 1.3 \cite{BPR} and so once again  we will refer several times to this paper for some details.\\
We begin to consider the case $\K=\C$. Let $g_1$, \ldots, $g_p$ be power series defining $(V,0)$. Let us replace $n$ by $n-1$ to assume that $(V,0)\subset (\C^{n-1},0)$ and let $(x_2, \ldots,
x_n)$ denote the  coordinates in $\C^{n-1}$. Let us set $g_0:=g$. 
After a linear change of coordinates in $x_2$, \ldots, $x_n$ (i.e preserving $x_1$)  we have that 
$$\prod_{m=0}^p(x_1-g_m (x_2,...,x_n))$$
 is $x_n$-regular. Thus we may write 
 $$
\prod_{m=0}^p(x_1-g_m(x_2,\ldots,x_n))= u_{n}  (x)(x_n^{p_n}+ \sum_{j=1}^{p_n} a_{n-1,j} (x^{n-1}) x_n^{p_n-j}), 
$$
where $u_{n}(0)\ne 0$ and $a_{n-1,j}(0)=0$.  We set 
$$f_{n} (x)= x_n^{p_n}+ \sum_{j=1}^{p_n} a_{n-1,j} (x^{n-1}) x_n^{p_n-j}$$ so that 
\begin{align}\label{new2}
u_n(x) f_n(x)=\prod_{m=0}^p(x_1-\sum_{k=2}^nx_kb_{m,k}(x_2,\ldots,x_n))
\end{align}
with $g_m=\sum_{k=2}^nx_kb_{m,k}$ for some power series $b_{m,k}$ since $g_m(0)=0$ for every $m$.
We denote by $b  \in\C \{ x\} ^{p(n-1)}$ (resp.  $a_{n-1} \in \C\{x^{n-1}\}^{ p_{n}}$) the vector of the coefficients $b_{m,k}$
(resp.  of the coefficients  $a_{n-1,j}$).\\
 Again we denote by $\Delta_{n,i}$ the generalized discriminants  of $f_n$ which are 
polynomials in $a_{n-1}$.  
  Let $j_n$ be the positive integer such that  
 $$\Delta_{n,i} ( a_{n-1} )\equiv 0 \qquad i<j_n,$$
and  $\Delta_{n,j_n}  ( a_{n-1} ) \not \equiv 0$.  
 After a linear change of coordinates $(x_2, \ldots,x_{n-1})$ we may write 

$$
 \Delta_{n,j_n} ( a_{n-1} ) =  u_{n-1}  (x^{n-1})  x_1^{q_{n-1}} (x_{n-1}^{p_{n-1}}+
  \sum_{j=1}^{p_{n-1}} a_{n-2,j} (x^{n-2}) x_{n-1}^{p_{n-1}-j} ) , 
$$
where $u_{n-1}(0)\ne 0$ and $a_{n-2,j}(0)=0$.  We set

$$f_{n-1} = x_{n-1}^{p_{n-1}}+
  \sum_{j=1}^{p_{n-1}} a_{n-2,j} (x^{n-2}) x_{n-1}^{p_{n-1}-j} $$
    and the vector 
of its coefficients $a_{n-2,j}$ is denoted  by $a_{n-2} \in \C\{x^{n-2}\}^{ p_{n-1}}$.  
Let $j_{n-1}$ be the positive integer such that 
$$\Delta_{n-1,k}(a_{n-2})\equiv 0\ \ \forall k<j_{n-1} \text{ and } \Delta_{n-1,j_{n-1}}(a_{n-2})\not\equiv 0.$$
  Then again we divide  $\Delta_{n-1,j_{n-1}}  $ 
by the maximal power of $x_1$ and, after a linear change of coordinates $(x_2, ...,x_{n-2})$, we denote  by $f_{n-2} ( x^{n-2})$ the associated 
Weierstrass polynomial.  

We carry on with this construction and  
 define a sequence of Weierstrass  polynomials $f_{i} (  x^i )$, $i=1, \ldots, n-1$, such that 
 $f_i= x_i^{p_i}+ \sum_{j=1}^{p_i} a_{i-1,j} (x^{i-1}) x_i^{p_i-j}  $ is the Weierstrass polynomial associated to the first non identically zero generalized discriminant $\Delta_{i,j_i} ( a_{i+1} )$ of $f_{i+1}$,  divided by the maximal power of $x_1$, 
where  $a_{i}= (a_{i,1} , \ldots , a_{i,p_{i}} )$:
  \begin{align}\label{polynew:f_i}
 \Delta_{i+1,j_{i+1}} ( a_{i} ) =  u_{i} (x^{i})  x_1^{q_i} (x_i^{p_i}+ \sum_{j=1}^{p_i} a_{i-1,j} (x^{i-1}) x_i^{p_i-j}  ) ,  
 \quad i=0,\ldots,n-1.  
\end{align}
 Thus the vector  
of power series $a_i$ satisfies  
  \begin{align}\label{discnew:i}
\Delta_{i+1,k} ( a_{i-1} )\equiv 0 \qquad k<j_{i+1}   ,  \quad i=0,\ldots,n-1. 
\end{align}
\\
Then we use Theorem \ref{nest_ploski} to see that there exist two sets of indeterminates  $z=(z_1,\ldots, z_s)$ and $t=(t_1,\ldots, t_r)$, an increasing function $\tau$,  convergent power series $z_i(x)\in\C \{x\}$ {vanishing at 0}, complex numbers $\tt_1$, \ldots, $\tt_r\in\C$, algebraic power series $u_i (t,x^i,z) \in \ovl\Q\lg t,x^{i},z_1,\ldots,z_{\tau(i)}\rg$ and  vectors of algebraic power series 
$$b(t,x,z)\in\ovl\Q\lg t,x,z\rg^{p(n-1)},$$
$$a_i(t,x^i,z) \in\ovl\Q\lg t,x^{(i)},z_1,\ldots,z_{\tau(i)}\rg^{p_i},$$
 such that the following holds:
\begin{enumerate}

\item[a)] $z_1(x),\ldots, z_{\tau(i)}(x)$ depend only on $(x_1,\ldots, x_{i})$,
\item[b)] $a_i(t,x^i,z),\ u_i (t,x^i,z),\ b(t,x,z)$ are solutions of  \eqref{new2}, \eqref{polynew:f_i} and
\eqref{discnew:i},
\item[c)] $a_i(x^i)= a_i(\tt,x^i,z(x^i))$,   $u_i(x^i)= u_i(\tt,x^i,z(x^i)), b(x)=b(\tt,x,z(x))$.\\
\end{enumerate}
Then we repeat what we did in the proof of Theorem \ref{thm1}.
Let $\ovl u_i$ be the constant coefficient of $u_i(x^i)$. Because $z(0)=0$ we have that $\ovl u_i=u_i(\tt,0,0)$ and $u_i(t,0,0)\in\ovl\Q\langle t\rangle$. In particular $u_i(\tt,0,0)\neq 0$.  Let $\g : \mathcal U\lgw \C^r$ be the analytic map defined by
$$\g(\la)=(1-\la)\q+\la \tt$$ where $\mathcal U$ is an  open connected neighborhood of the closed unit disc in $\C$ and $\q\in\ovl\Q^r$. Because $u_i(\tt,0,0)\neq 0$ and $\ovl\Q$ is dense in $\C$ we may choose $\q$ close enough to $\tt$ 
 such that 
$$u_i(\g(\la),0,0)\neq 0 \ \forall i, \ \forall \la\in \mathcal U.$$
Again because $\ovl\Q$ is dense in $\C$ we can find  $\q\in\ovl\Q$ close enough to $\tt$  such that the following are, for all $\la\in \mathcal U$, well defined convergent power series in $x$:
\begin{align*}
 F_n(\la,x) :=  x_n^{p_n}+ 
 \sum_{j=1}^{p_n} a_{n-1,j} (\g(\la),x^{n-1}, \la z(x^{n-1})) x_n^{p_n-j} ,
\end{align*}
 \begin{align*}
 F_i(\la,x) := x_i^{p_i}+ \sum_{j=1}^{p_i} a_{i-1,j} (\g(\la),x^{i-1}, \la z(x^{i-1}) ) x_i^{p_i-j} , \quad i=0, \ldots , n-1,
\end{align*}
\begin{align*} 
 u_i(\g(\la),x^i,z(x^i)), \quad i=1, \ldots , n-1.\end{align*}
We have that
$$u_n(\g(\la),x,\la z(x)) F_n(\la, x)=\prod_{m=0}^p(x_1-\sum_{k=2}^nx_kb_{m,k}(\g(\la),x,\la z(x))).$$
By the implicit function theorem or the Weierstrass Preparation Theorem
$$x_1-\sum_{k=2}^nx_kb_{m,k}(\g(\la),x,\la z(x))=v_m(\la,x)(x_1-G_m(\la,x_2,\ldots,x_n))$$
where $v_m(\la,x)\in\C\{\la,x\}$, $G_m(\la,x_2,\ldots,x_n)\in\C\{\la,x_2,\ldots,x_n\}$ and $v_m(0,0)\neq 0$. Because
$$x_1-\sum_{k=2}^nx_kb_{m,k}(\g(0),x, 0)\in\ovl\Q\lg x\rg$$
we have that
$$v_m(0,x),\ G_m(0,x_2,\ldots, x_n)\in\ovl\Q\lg x\rg$$
by unicity in the Weierstrass Preparation Theorem. We set 
$$\hat g_m(y):=G_m(0,y),\ m=0,\ldots,p$$
where $y=(y_1,\ldots,y_{n-1})$ is a new vector of indeterminates. \\
Then, for both cases $\K=\C$ or $\R$, we conclude exactly as in \cite{BPR} (see the end of 5.4 and Proposition 5.3 \cite{BPR}) to show that
there is a homeomorphism $h : (\K^n,0) \to (\K^n,0)$ such that $(h(V),0)$ is a germ of a Nash subset of $\K^n$ defined over $\ovl\Q\cap\K$ and $g\circ h$ 
 is the germ of a Nash function defined over $\ovl\Q\cap\K$.  \\
 \\
 Then we deduce from Proposition \ref{AM2} the following analogue of Theorem 5.4 \cite{BPR}:
 
 \begin{cor}\label{AM}
 Let $g_i$ be algebraic powers series with coefficients in $\ovl\Q\cap\K$ defining Nash function germs $g_i:(\K^n,0)\lgw (\K,0)$. Then there exist a Nash diffeomorphism $h:(\K^n,0)\lgw (\K^n,0)$ and Nash units $u_i:(\K^n,0)\lgw \K$, $u_i(0)\neq 0$, such that, for every $i$, $u_i(x)g_i(h(x))$ are polynomial function germs defined over $\ovl\Q\cap \K$. 
 \end{cor}
 
 \begin{proof}
We have the following fact: let $(Y,0)\subset (\K^n,0)$ be a Nash set germ defined by algebraic power series with coefficients in $\ovl\Q\cap\K$. Then there exists a Nash diffeomorphism $h:(\K^n,0)\lgw (\K^,0)$ such that for every irreducible analytic component $W$ of $(Y,0)$, the ideal of functions vanishing on $h(W)$ is generated by polynomials with coefficients in $\ovl\Q\cap\K$. Indeed this fact follows from Proposition \ref{AM2} by applying word for word the proof of Theorem 5 "(i) $\Longrightarrow$ (iv)" \cite{BK}.\\
Thus when $\K=\C$ this fact applied to the germ $(Y,0)$ defined by the products of the $g_i$ proves the theorem.\\
When $\K=\R$ we conclude as done for this case in the proof of Theorem 5.4 \cite{BPR}.
 
 \end{proof}

Let us recall that we have shown that there is a homeomorphism $h : (\K^n,0) \to (\K^n,0)$ such that $(h(V),0)$ is the germ of a Nash subset of $\K^n$ defined over $\ovl\Q\cap\K$ and $g\circ h$ 
 is the germ of a Nash function defined over $\ovl\Q\cap\K$. So we conclude the proof of Theorem \ref{thm2} by using  Theorem \ref{AM} and Theorem 5.5 \cite{BPR}.

\end{document}